\documentclass{amsart}
\usepackage[T1]{fontenc}
\usepackage[utf8]{inputenc}
\usepackage{verbatim}
\usepackage{amsfonts,amssymb,amsmath}
\usepackage{srcltx}

\usepackage{mathtools}
\usepackage{indentfirst}
\usepackage{enumerate}
\usepackage{cases}

\usepackage{fancyhdr}

\usepackage[square,sort,comma,numbers]{natbib}
\usepackage{graphicx}

\usepackage{hyperref}
\usepackage{bbold}

\newtheorem{thm}{Theorem}
\newtheorem{lem}{Lemma}

\newtheorem{cor}[thm]{Corollary}

\newtheorem{conj}{Conjecture}

\newcommand{\R}{\mathbb{R}}
\newcommand{\Z}{\mathbb{Z}}

\newtheorem{thmA}{Theorem}

\usepackage[square,sort,comma,numbers]{natbib}
\usepackage{graphicx}

\title[On two conjectures concerning squarefree numbers]{On two conjectures concerning squarefree numbers in arithmetic progressions}

\address{EPFL SB MATHGEOM TAN\\ Station 8\\ CH-1015 Lausanne\\ Switzerland}
\email{ramon.moreiranunes@epfl.ch}

\begin{document}
\author{Ramon M. Nunes} \thanks{This work was partially supported by the labex LMH through the grant no. ANR-11-LABX-0056-LMH in the “Programme des Investissements d'Avenir”.}




\begin{abstract}
We prove upper bounds for the error term of the distribution of squarefree numbers up to $X$ in arithmetic progressions $\!\!\pmod q$ making progress towards two well-known conjectures concerning this distribution and improving upon earlier results by Hooley. We make use of recent estimates for short exponential sums by Bourgain-Garaev and for exponential sums twisted by the M\"obius function by Bourgain and Fouvry-Kowalski-Michel.
\end{abstract}

\keywords{squarefree integers, arithmetic progressions, exponential sums}

\subjclass[2010]{Primary 11N37; Secondary 11L05}

\maketitle

\section{Introduction}

The distribution of arithmetic sequences in arithmetic progressions is a central subject in analytic number theory. Let $f:\Z_{>0}\rightarrow \R_{>0}$ be a positive arithmetic sequence. If $f$ is sufficiently reasonable, one expects that, for all $(a,q)=1$, we have

\begin{equation}\label{as}
\sum_{\substack{n\leq X\\ n\equiv a\!\!\!\!\!\pmod{q}}}f(n)\sim\frac{1}{\varphi(q)}\sum_{\substack{n\leq X\\ (n,q)=1}}f(n)\;\;\;\;\;\;\;\;(X\rightarrow \infty).
\end{equation}

Two important questions that arise naturally concern the uniformity of such a formula (see section \ref{range} below) and the size of the error term

\begin{equation}
E_f(X,q,a):=\sum_{\substack{n\leq X\\ n\equiv a\!\!\!\!\!\pmod{q}}}f(n)-\frac{1}{\varphi(q)}\sum_{\substack{n\leq X\\ (n,q)=1}}f(n).
\end{equation}
These questions are intimately related to the analytic properties of the associated $L$-functions

$$
L_f(s,\chi):=\sum_{n=1}^{\infty}f(n)\chi(n)n^{-s},
$$
where $\chi$ is a Dirichlet character. Two particularly interesting cases occur when $f=\Lambda$ is the van Mangoldt function or $f=\mu^2$, where $\mu$ is the M\"obius function. The first one is closely related to the distribution of prime numbers while the latter corresponds to squarefree numbers. For such choices the associated $L$-functions are, respectively,

$$
L_{\Lambda}(s,\chi)=\frac{L'(s,\chi)}{L(s,\chi)}\:\:\:\text{and}\:\:\:L_{\mu^2}(s,\chi)=\frac{L(s,\chi)}{L(s,\chi^2)},
$$
where $L(s,\chi)$ is the classical Dirichlet $L$-function. In that context, H. L. Montgomery stated two conjectures whose implications are much deeper than the generalized Riemann Hypothesis for Dirichlet $L$-functions (GRH). The original conjectures can be found, respectively in \citep[Formula (15.9), page 136]{montgomery1971topics} and \citep[top of page 145]{croft1975square}. We state them in slightly improved forms, which, in the case of the van Mangoldt function, is due to Friedlander and Granville \citep{friedlander1989limitations} and in the case of squarefree numbers can be found in a recent preprint by Le Boudec \citep{boudec2014distribution}.

\begin{conj}
Let $\epsilon,X>0$. Let $a$ and $q$ be integers such that $(a,q)=1$ and $1\leq q\leq X$, then we have
\begin{itemize} 
\item for primes,
\begin{equation}\label{Primes}
E_{\Lambda}(X,q,a)\ll_{\epsilon}X^{\epsilon}\left(\frac Xq\right)^{1/2},
\end{equation}
\end{itemize}
\begin{itemize}  
\item for squarefree numbers,
\begin{equation}\label{Sfree}
E_{\mu^2}(X,q,a)\ll_{\epsilon}X^{\epsilon}\left(\frac Xq\right)^{1/4},
\end{equation}
\end{itemize}  
where the implied constants depend at most on $\epsilon$.

\end{conj}

Concerning \eqref{Primes}, it is not known, for the moment, if there exists $\delta>0$ such that

$$
E_{\Lambda}(X,q,a)\ll\left(\dfrac{X}{q}\right)^{1-\delta}
$$
holds in any range whatsoever. For instance, GRH  would imply that for every $\alpha<{\frac12}$, there exists $\delta=\delta(\alpha)>0$, such that

$$
E_{\Lambda}(X,q,a)\ll \left(\dfrac{X}{q}\right)^{1-\delta},
$$
uniformly for $1\leq q\leq X^{\alpha}$ and $(a,q)=1$. Still under GRH, Tur\'an \citep{turan1936primzahlen} proved that \eqref{Primes} holds on average: Uniformly for $1\leq q\leq X$, we have

$$
\sum_{\substack{a=0\\(a,q)=1}}^{q-1}E_{\Lambda}(X,q,a)^2\ll X(\log X)^4.
$$
If one seeks for unconditional results, one needs an extra sum over $q\leq Q$, for a certain $Q\leq X$. In this direction, Montgomery proved that for every $A>0$, one has
\begin{equation}\label{BDH}
\sum_{q\leq Q}\sum_{\substack{a=0\\(a,q)=1}}^{q-1}E_{\Lambda}(X,q,a)^2= QX\log X+O_{A}\left(QX\log \frac{2X}{Q}+QX(\log X)^{-A}\right), 
\end{equation}
uniformly for $Q\leq X$, where the implied constant depends at most on $A$. Therefore, for every $A>0$, \eqref{BDH} implies that \eqref{Primes} holds true on average over $q\leq X(\log X)^{-A}$ and $a$~modulo~${q}$, with $(a,q)=1$.

In the case of squarefree numbers, Prachar \citep{Prachar1958kleinste} proved that for every $\epsilon>0$, uniformly for $(a,q)=1$, $X\geq 2$,

$$
E_{\mu^2}(X,q,a)\ll_{\epsilon} X^{\frac{1}{2}+\epsilon}q^{-\frac{1}{4}}+q^{\frac{1}{2}+\epsilon},
$$
where the implied constant depends at most on $\epsilon$. This result was later improved by Hooley \citep{hooley1975note} who showed that the above error term can be replaced by

\begin{equation}\label{hoo-C3}
O_{\epsilon}\left(X^{\frac{1}{2}}q^{-\frac{1}{2}}+q^{\frac{1}{2}+\epsilon}\right).
\end{equation}
Both this results show that for every $\epsilon>0$, there exists $\delta=\delta(\epsilon)>0$
$$	
E_{\mu^2}(X,q,a)\ll_{\epsilon} \left(\frac{X}{q}\right)^{1-\delta},
$$
uniformly for $q\leq X^{\frac{2}{3}-\epsilon}$. In a parallel to T\'uran's result, the author \citep{nunes2015square} (unconditionally) proved  the asymptotic formula:
$$
\sum_{\substack{a=0\\(a,q)=1}}^{q-1}E_{\mu^2}(X,q,a)^2\sim C\prod_{p\mid q}\left(1+{\frac{2}{p^2}}\right)^{-1}X^{\frac1 2}q^{\frac12},
$$
as $X,q\rightarrow\infty$ and $q$ satisfying $X^{{\frac{31}{41}}+\epsilon}\leq q\leq X^{1-\epsilon}$, where $C$ is an absolute positive constant. This implies \eqref{Sfree} on average on the above mentioned range. In a subsequent work, le Boudec \cite {boudec2014distribution} proves that if one seeks for an upper bound rather than an asymptotic formula, one gets a larger range. He proved that
$$
\sum_{\substack{a=0\\(a,q)=1}}^{q-1}E_{\mu^2}(X,q,a)^2\ll_{\epsilon} X^{{\frac12}+\epsilon}q^{\frac12},
$$
uniformly for $X^{\frac12}\leq q\leq X$.

Suppose $q\leq X^{\frac12}$. In this case Hooley's result shows an approximation of \eqref{Sfree}, with $1/2$ instead of $1/4$. In this paper, we improve on the exponent $1/2$ for $X^{\epsilon}\leq q\leq X^{\frac12-\epsilon}$, therefore making a further step towards \eqref{Sfree}. For simplicity, we shall henceforth write $E(X,q,a)$ instead of $E_{\mu^2}(X,q,a)$.
\begin{thm}\label{mainFKM}
For every $\epsilon>0$, we have the inequality
$$
E(X,q,a) \ll_{\epsilon} \left(\dfrac{X}{q}\right)^{\frac{1}{2}+\epsilon}q^{-\frac{1}{192}}+\left(\dfrac{X}{q}\right)^{\frac{24}{49}+\epsilon}q^{\frac{3}{196}},
$$
uniformly for every $X>1$, every prime number $q$ such that $q\leq X$, and every integer $a$ such that $(a,q)=1$.
\end{thm}
Note that for $X^{\alpha}<q<2X^{\alpha}$, $\alpha\leq {\frac25}$, Theorem \ref{mainFKM} shows that
$$
E(X,q,a) \ll_{\epsilon}\left(\dfrac{X}{q}\right)^{\frac{1}{2}-\frac{\alpha}{192(1-\alpha)}+\epsilon}+\left(\dfrac{X}{q}\right)^{\frac{1}{2}-\frac{2-5\alpha}{196(1-\alpha)}+\epsilon}.
$$
For example, when $\alpha={\frac{96}{283}}<{\frac25}$, the exponent is ${\frac{93}{187}}<{\frac12}$, which is the best exponent given by Theorem \ref{mainFKM}. Therefore, we have the following corollary

\begin{cor}\label{main-cor}
For every $0<\epsilon<{\frac15}$, let $\delta(\epsilon)=\min\left(\frac{\epsilon}{192(1-\epsilon)},\frac{25\epsilon}{196(3+5\epsilon)}\right)$. Then we have
$$
E(X,q,a) \ll_{\epsilon} \left(\dfrac{X}{q}\right)^{\frac{1}{2}-\delta(\epsilon)},
$$
uniformly for every $X>1$ and every prime number $q$ such that $X^{\epsilon}\leq q\leq X^{\frac{2}{5}-\epsilon}$, and every integer $a$ such that $(a,q)=1$.
\end{cor}
Our next result states that we can still obtain an exponent below $ \frac{1}{2}$ for $q$ as large as $X^{{\frac12}-\epsilon}$, but one cannot quantify the exponent.

\begin{thm}\label{mainBourg}
For every $0<\eta<{\frac14}$, there exists $\delta=\delta(\eta)>0$ such that we have

$$
E(X,q,a) \ll_{\eta} \left(\dfrac{X}{q}\right)^{\frac{1}{2}-\delta},
$$
uniformly for every $X>1$ and every prime number $q$ such that $X^{\eta}\leq q\leq X^{\frac{1}{2}-\eta}$, and integer $a$ such that $(a,q)=1$.
\end{thm}
The main new input in the proofs of Theorems \ref{mainFKM} and \ref{mainBourg} are bounds for exponentials sums twisted by the M\"obius function given by Fouvry \textit{et al.} \citep{fouvry2014algebraic} and Bourgain \citep{bourgain2005more}. The same exponential sums were estimated trivially in \citep{hooley1975note}.

\subsection{Range of uniformity}\label{range}

Concerning the range of uniformity, it is largely believed that if $f$ is sufficiently reasonable, then \eqref{as} should hold uniformly for $q\leq X^{1-\epsilon}$. In the case where $f=\mu^2$, we know, thanks to Prachar \citep{Prachar1958kleinste} that this is true in the range $q\leq X^{{\frac23}-\epsilon}$. Our next result proves that we can overcome the threshold $q=X^{\frac23}$, but only by a small power of $\log X$. More precisely, we prove
\begin{thm}\label{2/3}
For every $0<\gamma<\frac{1}{2}$, there exists $C(\gamma)$ such that for every $X>3$, for every prime $q$, for every $a$ coprime with $q$, one has the inequality
$$
\left|E(X,q,a)\right| \leq C(\gamma)\left( \frac{X^{\frac13}(\log\log X)^{\frac73}}{(\log X)^{{\frac16}-{\frac{\gamma}{3}}}}+\frac{X(\log\log X)^2}{q(\log X)^{\frac{\gamma}{2}}}\right).
$$
\end{thm}
By taking $\gamma$ arbitrarily small, we see that for every $\epsilon>0$, the asymptotic formula
\begin{equation}\label{as-23}
\sum_{\substack{n\leq X\\ n\equiv a\!\!\!\!\!\pmod{q}}}\mu^2(n)\sim\frac{1}{\varphi(q)}\sum_{\substack{n\leq X\\ (n,q)=1}}\mu^2(n)\left(\sim \frac{6}{\pi^2}\left(1-\frac{1}{p^2}\right)^{-1}\frac Xq\right)
\end{equation}
holds as $X\rightarrow \infty$, uniformly for $q\leq X^{\frac23}(\log X)^{\frac16 -\epsilon}$ and integers $(a,q)=1$. This is the first time ocurrence of such an asymptotic formule for \textit{prime} values of $q$ tending to infinity faster than $X^{\frac23}$.
The main ingredient in the proof of Theorem \ref{2/3} is an upper bound for incomplete Kloosterman sums that fall just outside the Polya-Vinogradov range given by Bourgain-Garaev (see Theorem \ref{Tbourgain}).
Finally, we remark that Corollary \eqref{as-23} implies that if $n(a,q)$ denotes the least positive squarefree number that is congruent to $a$ modulo $q$, then for every $\epsilon>0$, we have the inequality
$$
n(a,q)\ll_{\epsilon} q^{\frac 32}(\log q)^{-\frac 16 +\epsilon},
$$
where the implied constant depends only on $\epsilon$. The best result in this direction is due to Heath-Brown \citep{heath1982least}, who proved that
$$
n(a,q)\ll(d(q)\log q)^6q^{\frac{13}9}.
$$
%
%
%
%

\section{Preliminary results}

The next lemma is a simple consequence of Weil's bound for exponential sums that come from algebraic curves over finite fields and classical estimates for Gauss sums.
\begin{lem}\label{Weil-C3}
Let $A<B$ be real number, let $a$ and $q$ be integers satisfying $(a,q)=1$, $q\geq 1$. Then for every $\epsilon>0$, we have
$$ 
\sum_{A<r\leq B}e\left(\frac{a{\bar r}^2}{q}\right)\ll_{\epsilon} q^{\epsilon}\left(\frac{B-A}{q^{\frac{1}{2}}}+q^{\frac{1}{2}}\right)
$$
where ${\bar r}$ denotes the multiplicative inverse of $r\pmod{q}$, $e(z)=e^{2\pi iz}$ and the implicit constant depends at most on $\epsilon$.
\end{lem}

\subsection{Approximation to the $\psi$ function}
The next lemma is an useful analytic tool to avoid the problems arising from the lack of continuity of the sawtooth function. The version we use here can be found in \citep{fouvry1980theorem}, and is inspired by an idea of Vinogradov.

\begin{lem}\label{FoIw}(see \citep[Lemma 4]{fouvry1980theorem})
Let $\psi(x)= x-\left\lfloor x \right\rfloor -\frac{1}{2}$ and $Y>1$. There are two functions $A$ and $B$ with period 1 such that, for every real $x$, one has

$$
\left|\psi(x)-A(x)\right|\leq B(x),
$$
where
$$
A(x)=\sum_{h\neq 0}A_he(hx),
$$
$$
B(x)=Y^{-1}+\sum_{h\neq 0}B_he(hx),
$$
with
\begin{equation}\label{AhBh<}
A_h,B_h\ll C_h:=\min\left(\frac{1}{|h|},\frac{Y^3}{|h|^4}\right),\;(h\neq 0).
\end{equation}
\end{lem}

\subsection{Exponential sums twisted by the M\"obius function}

We now state the estimates for exponential sums twisted by the M\"obius functions that were mentioned in the introduction. The first one is a very particular case of \citep[Theorem 1.7]{fouvry2014algebraic} by Fouvry, Kowalski and Michel, which was based on a previous work by Fouvry and Michel \citep{fouvry1998certaines}. It gives non-trivial bounds for $R\geq q^{{\frac34}+\epsilon}$ and will be used in several places of the proof of Theorem \ref{mainFKM}.

\begin{thmA}\label{TFKM}
For every $\epsilon>0$, there exists $C(\epsilon)$ such that, for every $R\geq 1$, for every prime $q$, and every $a$ coprime with $q$, one has the inequality

\begin{equation}\label{eq-tFKM}
\left|\sum_{n\leq R}\mu(n)e\left(\dfrac{a\overline{n}^2}{q}\right)\right|\leq C(\epsilon)R\left(1+\frac{q}{R}\right)^{\frac1{12}}q^{-\frac{1}{48}+\epsilon}.
\end{equation}
\end{thmA}
To prove Theorem \ref{mainBourg}, we need to replace Theorem \ref{TFKM} by an estimate that gives something non-trivial in the larger range $R\geq q^{\frac{1}{2}+\epsilon}$. For this we have the next result which is a combination of a remarkable result by Bourgain \citep{bourgain2005more}, which is non-trivial in the range $q^{\frac{1}{2}+\epsilon}\leq R\leq q$, and Theorem \ref{TFKM} itself.

\begin{thmA}\label{Tbourgain}
For every $\eta>0$, there exists $\delta(\eta)>0$ and $C(\eta)$ such that, for every $R\geq 1$, for every prime $q$ satisfying $q^{\frac{1}{2}+\eta}\leq R\leq q^{\frac1{\eta}}$, and every $a$ coprime with $q$, one has

\begin{equation}\label{eq-tbourgain}
\left|\sum_{n\leq R}\mu(n)e\left(\dfrac{a\overline{n}^2}{q}\right)\right|\leq C(\eta)R^{1-\delta(\eta)}.
\end{equation}
\end{thmA}
\begin{proof}
For $q^{\frac{1}{2}+\eta}\leq R\leq q$, this is exactly \cite[Theorem A.9]{bourgain2005more}. For $q<R\leq q^{\frac{1}{\eta}}$, it follows from Theorem  \ref{TFKM}.
\end{proof}

\subsection{Short exponential sums}

In the course of the proof of Theorem \ref{2/3}, we are led to deal with very short exponential sums, for which we use the following result by Bourgain-Garaev. It gives non-trivial results for short Kloosterman sums $\!\!\pmod{q}$ where the length is as small as $q^{\epsilon}$ for any $\epsilon>0$.

\begin{thmA}(see \citep[Theorem 16]{bourgain2014sumsets})\label{TBG}
There exists an absolute constant $C$ such that for every $M\geq2$, every prime $q$, and every $a$ coprime with $q$, one has the inequality

\begin{equation*}
\left|\sum_{m\leq M}e\left(\dfrac{a\overline{m}}{q}\right)\right|\leq C\frac{M\log q(\log\log q)^{3}}{(\log M)^{\frac32}},
\end{equation*}

\end{thmA}

The following lemma is obtained by combining Lemma \ref{FoIw} and Theorem \ref{TBG} above.

\begin{lem}\label{boundpsi}
Let $\psi(x)$ be as in Lemma \ref{FoIw}. Then we have the inequality

$$
\sum_{m\leq M}\psi\left(N+\frac{a\overline{m}}{q}\right) \ll \frac{M}{\log q}+\frac{M\log q(\log\log q)^4}{(\log M)^{\frac{3}{2}}}.
$$
uniformly for every pair of real numbers $M,N$ such that $M>1$ and prime $q>2$, where the implied constant is absolute.

\end{lem}

\begin{proof}

Let $1\leq Y\leq q$ to be chosen later. By Lemma \ref{FoIw}, we deduce that

$$
\sum_{m\leq M}\psi\left(N+\frac{a{\bar m}}{q}\right)\ll \frac{M}{Y}+M\sum_{\substack{h=-\infty\\h\neq 0}}^{\infty}C_{hq}+\sum_{q\nmid h}C_h\left|\sum_{m\leq M}e\left(\frac{ah{\bar m}}{q}\right)\right|,
$$
where $C_h$ is as in \eqref{AhBh<}. By Theorem \ref{TBG} and the bounds \eqref{AhBh<}, we have that

$$
\sum_{m\leq M}\psi\left(N+\frac{a{\bar m}}{q}\right)\ll \frac{M}{Y}+\frac{M\log q(\log\log q)^3}{(\log M)^{\frac{3}{2}}}\log Y,
$$
since $Y\leq q$. We conclude by choosing $Y=\log q$.
\end{proof}

\subsection{Selberg's Sieve}

Another important input to the proof of Theorem \ref{2/3} is the Selberg sieve for detecting squares (see \citep[Chapter 8]{friedlander2010opera}). We shall need the following result.

\begin{thmA}\label{SLB}
Let $\mathcal{A}=(a_n)$ be a finite sequence of non-negative numbers. Let $P$ be a squarefree number. For each $p\mid P$, let $\Omega_p$ be a set of congruence classes modulo $p$. For every $d\mid P$ we write
< 
\begin{equation}\label{Ad}
 |\mathcal{A}_d|=\sum_{\substack{n\!\!\!\pmod{d}\in \Omega_p\\ \text{for every }p\mid d}}a_n=g(d)Y+r_d(\mathcal{A}),
\end{equation}
where $Y>0$ and $g(d)$ is a multiplicative function with $0<g(p)<1$ for $p\mid P$. Let $h(d)$ be the multiplicative function given by $h(p)=g(p)(1-g(p))^{-1}$ and for any $D>1$ define

$$
J=J(D):=\sum_{\substack{d\mid P\\d<\sqrt{D}}}h(d).
$$
Then, for any $D>1$ we have the inequality

$$
\sum_{\substack{n\!\!\!\pmod{d}\not\in \Omega_p\\ \text{for every }p\mid P}}a_n\leq YJ^{-1} + \sum_{\substack{d\mid P\\d\leq \sqrt{D}}}\tau_3(d)\left|r_d(\mathcal{A})\right|,
$$
where $\tau_3$ is the generalized divisor function.
\begin{proof}
The proof follows exactly as that of \citep[Theorem 7.1]{friedlander2010opera}, taking into account the simple inequality

$$
\sum_{\substack{n\!\!\!\pmod{p}\not\in \Omega_p\\ \text{for every }p\mid P}}a_n\:\leq\: \sum_{n}a_n\bigg(\sum_{\substack{d\\n\!\!\!\!\!\pmod{p}\in \Omega_p\\ \text{for every }p\mid d}}\rho_d\bigg)^2,
$$
for any real numbers $\rho_d$ supported on $d\mid P$ with $\rho_1=1$. the optimal choice of these $\rho_d$ is the heart of the Selberg's sieve.

\end{proof}

\end{thmA}

\section{Proofs of the results}

\subsection{Proof of Theorem \ref{mainFKM}}
Let $X>1$ and $\epsilon>0$ be real numbers and let $q$ be a prime number. Since the upper bound given by Theorem \ref{mainFKM} is worse than \eqref{hoo-C3} for $q\leq X^{2/5}$, we may suppose $q\leq X^{2/5}$. Let

\begin{equation}\label{S=}
S:=\sum_{(r,q)=1}\mu(r)\sum_{\substack{m\leq X/r^2\\m\equiv a\overline{r}^2\!\!\!\!\pmod q}}1, 
\end{equation}
and

\begin{equation}\label{S0=}
S_0:=\frac{1}{\varphi(q)}\sum_{\substack{n\leq X\\(n,q)=1}}\mu^2(n), 
\end{equation}

We have

\begin{align}\label{S-S0}
E(X,q,a)=S-S_0.
\end{align}
It is rather elementary to see that $S_0$ satisfies (recall that $q$ is prime)

\begin{equation}\label{S0}
S_0=\dfrac{6}{\pi^2}\left(1-\dfrac{1}{q^2}\right)^{-1}\dfrac{X}{q}+O_{\epsilon}(X^{\frac12}q^{-1+\epsilon}).
\end{equation}
Let $1< R\leq X^{\frac12}$ be a parameter to be chosen later depending on $X$ and $q$. We split $S$ as

\begin{equation}\label{I+II}
S=S_{I}+S_{II},
\end{equation}
where
\begin{equation}\label{R-3}
\begin{cases}
S_{I}=\displaystyle\sum_{\substack{r\leq R\\(r,q)=1}}\mu(r)\!\!\!\!\!\!\displaystyle\sum_{\substack{m\leq X/r^2\\m\equiv a\overline{r}^2\!\!\!\!\pmod q}}\!\!\!\!1,\\
S_{II}=\displaystyle\sum_{\substack{R<r\leq X^{\frac12}\\(r,q)=1}}\mu(r)\!\!\!\!\!\!\displaystyle\sum_{\substack{m\leq X/r^2\\m\equiv a\overline{r}^2\!\!\!\!\pmod q}}\!\!\!\!1.
\end{cases}
\end{equation}
For the first sum in \eqref{R-3}, we have that

\begin{align}\label{S=TUV}
S_I&=\displaystyle\sum_{\substack{r\leq R\\(r,q)=1}}\mu(r)\left\{\frac{X}{qr^2}-\psi\left(\dfrac{X}{qr^2}-\dfrac{a{\bar r}^2}{q}\right)+\psi\left(-\dfrac{a{\bar r}^2}{q}\right)\right\}\notag\\
&=:\mathcal{T}-\mathcal{U}+\mathcal{V},
\end{align}
say. The first term satisfies

\begin{equation}\label{T=}
\mathcal{T}=\dfrac{6}{\pi^2}\left(1-\dfrac{1}{q^2}\right)^{-1}\dfrac{X}{q}+ O\left(R^{-1}Xq^{-1}\right).
\end{equation}

\subsubsection{Study of $\mathcal{V}$}
Let $1<Y\leq X$ be a parameter to be chosen optimally later depending on $X$ and $q$, then Lemma \ref{FoIw} gives us two functions $A$ and $B$ whose Fourier coefficients satisfy \eqref{AhBh<}, and such that

\begin{equation}\label{V1V2}
|\mathcal{V}|\leq \mathcal{V}_1 + \mathcal{V}_2,
\end{equation}
where

$$
\mathcal{V}_1=\left|\displaystyle\sum_{\substack{r\leq R\\(r,q)=1}}\mu(r)A\left(-\dfrac{a{\bar r}^2}{q}\right)\right|,\;\mathcal{V}_2=\displaystyle\sum_{\substack{r\leq R\\(r,q)=1}}B\left(-\dfrac{a{\bar r}^2}{q}\right).
$$
Writing down the Fourier development for $A(x)$, and using \eqref{AhBh<}, we see that

\begin{equation}\label{a}
\mathcal{V}_1\leq \sum_{h\neq 0}C_h\left|\sum_{\substack{r\leq R\\(r,q)=1}}\mu(r)e\left(-\frac{ah{\bar r}^2}{q}\right)\right|.
\end{equation}
The contribution of the terms where $q\mid h$ is trivially seen to be

\begin{equation}\label{b}
\ll R\sum_{h\neq 0}C_{hq}\ll_{\epsilon} X^{\epsilon}Rq^{-1}, 
\end{equation}
by \eqref{AhBh<}. For the remaining terms, we use Theorem \ref{TFKM} and see that their contribution is

\begin{equation}\label{c}
\ll_{\epsilon} X^{\epsilon}\left(Rq^{-\frac{1}{48}}+R^{\frac{11}{12}}q^{\frac{1}{16}}\right),
\end{equation}
again by \eqref{AhBh<}. Hence, by \eqref{a}, \eqref{b} and \eqref{c}, we have

\begin{equation}\label{V1}
\mathcal{V}_1\ll_{\epsilon}X^{\epsilon}\left(Rq^{-\frac{1}{48}}+R^{\frac{11}{12}}q^{\frac{1}{16}}\right).
\end{equation}
The analysis of $\mathcal{V}_2$ is completely analogous. The only difference is that we shall need Lemma \ref{Weil-C3} instead of Theorem \ref{TFKM}. We obtain

\begin{equation}\label{V2}
\mathcal{V}_2\ll_{\epsilon}X^{\epsilon}\left(Rq^{-\frac{1}{2}}+q^{\frac{1}{2}}+RY^{-1}\right).
\end{equation}
Gathering \eqref{V1} and \eqref{V2} in \eqref{V1V2}, we have that

\begin{equation}\label{V-3}
 \mathcal{V}\ll_{\epsilon}X^{\epsilon}\left(Rq^{-\frac{1}{48}}+R^{\frac{11}{12}}q^{\frac{1}{16}}+q^{\frac{1}{2}}+RY^{-1}\right).
\end{equation}
\subsubsection{Study of $\mathcal{U}$}
This part is very similar to the study of $\mathcal{V}$ but with the difference that we need an Abel summation to take care of the oscillation of the term $X/qr^2$. Let $1<R_0\leq (X/q)^{\frac{1}{2}}$ to be chosen optimally later. We write

\begin{equation}\label{U=W+R_0}
\mathcal{U}=\mathcal{W}+O(R_0),
\end{equation}
where

$$
\mathcal{W}:=\sum_{\substack{R_0<r\leq R\\(r,q)=1}}\mu(r)\psi\left(\dfrac{X}{qr^2}-\dfrac{a{\bar r}^2}{q}\right).
$$
Again by Lemma \ref{FoIw}, we obtain two functions $A$ and $B$ satisfying \eqref{AhBh<} and such that

\begin{equation}\label{W1W2}
|\mathcal{W}|\leq \mathcal{W}_1 + \mathcal{W}_2,
\end{equation}
where

$$
\mathcal{W}_1=\left|\sum_{\substack{R_0<r\leq R\\(r,q)=1}}\mu(r)A\left(\dfrac{X}{qr^2}-\dfrac{a{\bar r}^2}{q}\right)\right|,\;\mathcal{W}_2=\sum_{\substack{R_0<r\leq R\\(r,q)=1}}B\left(\dfrac{X}{qr^2}-\dfrac{a{\bar r}^2}{q}\right).
$$
We write down the Fourier development of $A(x)$ and again, we separate the contribution from the terms where $q\mid h$ as we did for $\mathcal{V}_1$. We deduce

\begin{equation}\label{firstW1}
\mathcal{W}_1\leq \sum_{(h,q)=1}C_h\left|\sum_{\substack{R_0<r\leq R\\(r,q)=1}}\mu(r)e\left(\dfrac{hX}{qr^2}-\frac{ah{\bar r}^2}{q}\right)\right| +O_{\epsilon}\left(X^{\epsilon}Rq^{-1}\right). 
\end{equation}
Summing by parts, we see that

$$
\sum_{\substack{R_0<r\leq R\\(r,q)=1}}\mu(r)e\left(\dfrac{hX}{qr^2}-\frac{ah{\bar r}^2}{q}\right)\ll \dfrac{|h|X}{q}\sum_{\substack{R_0<t\leq R\\(t,q)=1}}\dfrac{1}{t^3}S(t,q) +S(R,q)+S(R_0,q),
$$
where
\begin{equation}\label{Stq}
S(t,q):=\max_{(a,q)=1}\left|\sum_{\substack{r\leq t\\(r,q)=1}}\mu(r)e\left(\frac{a{\bar r}^2}{q}\right)\right|
\end{equation}

The terms on the right-hand side of the above inequality can be estimated by means of Theorem \ref{TFKM} giving

\begin{multline*}
\sum_{\substack{R_0<r\leq R\\(r,q)=1}}\mu(r)e\left(\dfrac{hX}{qr^2}-\frac{ah{\bar r}^2}{q}\right)\ll_{\epsilon} X^{\epsilon}\left(|h|R_0^{-1}Xq^{-\frac{49}{48}}+|h|R_0^{-\frac{13}{12}}Xq^{-{\frac{15}{16}}}\right.\\
\left.+Rq^{-\frac{1}{48}}+R^{\frac{11}{12}}q^{\frac1{16}}+R_0\right).
\end{multline*}
Injecting it in \eqref{firstW1} and using \eqref{AhBh<}, we deduce

\begin{equation}\label{W1}
\mathcal{W}_1\ll_{\epsilon} X^{\epsilon}\left(R_0^{-1}XYq^{-{\frac{49}{48}}}+R_0^{-\frac{13}{12}}XYq^{-{\frac{15}{16}}}+Rq^{-\frac{1}{48}}+R^{\frac{11}{12}}q^{\frac{1}{16}}+R_0\right).
\end{equation}
The treatment of $\mathcal{W}_2$ goes in a similar fashion, replacing Theorem \ref{TFKM} by Lemma \ref{Weil-C3} in the appropriate places. We end up with

\begin{equation}\label{W2}
\mathcal{W}_2\ll_{\epsilon} X^{\epsilon}\left(R_0^{-1}XYq^{-\frac{3}{2}}+R_0^{-2}XYq^{-\frac{1}{2}}+Rq^{-\frac{1}{2}}+q^{\frac12}+RY^{-1}+R_0\right).
\end{equation}
Gathering \eqref{W1} and \eqref{W2} in \eqref{W1W2}, we have that

\begin{multline}\label{W}
\mathcal{W}\ll_{\epsilon}X^{\epsilon}\left(R_0^{-1}XYq^{-{\frac{49}{48}}}+R_0^{-\frac{13}{12}}XYq^{-{\frac{15}{16}}}+Rq^{-\frac{1}{48}}+R^{\frac{11}{12}}q^{\frac{1}{16}}+R_0^{-2}XYq^{-\frac{1}{2}}\right.\\
\left.+q^{\frac12}+RY^{-1}+R_0\right).
\end{multline}
Putting together \eqref{S=TUV}, \eqref{T=}, \eqref{V-3}, \eqref{U=W+R_0} and \eqref{W}, we see that

\begin{multline}\label{SIfinal}
 S_I=\dfrac{6}{\pi^2}\left(1-\dfrac{1}{q^2}\right)\dfrac{X}{q}+ O_{\epsilon}\left(X^{\epsilon}\left(R^{-1}Xq^{-1}+Rq^{-\frac{1}{48}}+R^{\frac{11}{12}}q^{\frac{1}{16}}+q^{\frac{1}{2}}+RY^{-1}\right.\right.\\
 \left.\left.+R_0^{-1}XYq^{-{\frac{49}{48}}}+R_0^{-\frac{13}{12}}XYq^{-{\frac{15}{16}}}+R_0^{-2}XYq^{-\frac{1}{2}}+R_0\right)\right).
\end{multline}
\subsubsection{Study of $S_{II}$}
We procceed  now to estimate $S_{II}$. We follow the lines of \citep[Lemma 2]{hooley1975note}. Ignoring the oscillation of $\mu$, we see that 

\begin{align}\label{SII-int}
S_{II}&\ll \sum_{R<r\leq \sqrt{X}}\sum_{\substack{m\leq\frac{X}{r^2}\\r^2m\equiv a\!\!\!\!\pmod{q}}}\log\left(3\sqrt{\dfrac{X}{r^2m}}\right)\notag\\
&= \sum_{R<r\leq \sqrt{X}}\sum_{\substack{m\leq\frac{X}{r^2}\\r^2m\equiv a\!\!\!\!\pmod{q}}}\displaystyle\int_{r}^{3\sqrt{{\frac{X}{m}}}}\dfrac{dt}{t}\notag\\
&\ll \displaystyle\int_{R}^{3\sqrt{X}}\sum_{m\leq\frac{9X}{t^2}}\sum_{\substack{r\leq t\\ r^2m\equiv a\!\!\!\!\pmod{q}}}1\dfrac{dt}{t}.
\end{align}
We put
\begin{equation}
Z=\max(R,q)
\end{equation}
and break up the integral on the right-hand-side from \eqref{SII-int} as

$$\displaystyle\int_{R}^{Z}\sum_{m\leq\frac{9X}{t^2}}\sum_{\substack{r\leq t\\ r^2m\equiv a\!\!\!\!\pmod{q}}}1\dfrac{dt}{t}+\displaystyle\int_{Z}^{3\sqrt{X}}\sum_{m\leq\frac{9X}{t^2}}\sum_{\substack{r\leq t\\ r^2m\equiv a\!\!\!\!\pmod{q}}}1\dfrac{dt}{t}$$
For the first integral, we use additive characters to detect the congruence condition. We have

\begin{equation}\label{Haha}
 \sum_{m\leq\frac{9X}{t^2}}\sum_{\substack{r\leq t\\ r^2m\equiv a\!\!\!\!\pmod{q}}}1=\dfrac{1}{q^{2}}\sum_{\alpha=0}^{q-1}\sum_{\beta=0}^{q-1}\mathcal{S}(q;\alpha,a\beta)\Theta(t,\alpha)\Theta\left(\frac{9X}{t^2},\beta\right),
\end{equation}
where
$$
\mathcal{S}(q;\alpha,\beta):=\sum_{h=1}^{q-1}e\left({\frac{\alpha h+\beta {\bar h}^2}{q}}\right),
$$
and
\begin{equation}\label{Theta}
\Theta(t,\alpha):=\sum_{n\leq t}e\left(-\frac{\alpha n}{q}\right)\ll \min\left(t,\left\|{\frac{\alpha}{q}}\right\|^{-1}\right).
\end{equation}
In order to estimate the sum on the right-hand side of \eqref{Haha}, we need bounds for $S(q;\alpha,\beta)$. In the cases where $\alpha\beta\equiv 0\pmod{q}$, the sum is either trivial, a Ramanujan sum or a Gauss sum. And the classical upper-bound for these sums are used. If both $\alpha$ and $\beta$ are $\not\equiv 0\pmod{q}$, then we shall use the following upper-bound that follows from the work of Weil
$$
S(q;\alpha,\beta)\ll q^{\frac12},\;(\alpha,\beta\not\equiv0\!\!\!\!\!\pmod{q}).
$$
Combining these bounds with \eqref{Theta}, we see that
\begin{align}\label{bound-blomer}
\sum_{m\leq\frac{9X}{t^2}}\sum_{\substack{r\leq t\\ r^2m\equiv a\!\!\!\!\pmod{q}}}1&\ll \left(\dfrac{X}{qt}\! + \frac{t}{q^{\frac32}}\sum_{\beta=1}^{q-1}\left\|{\frac{\beta}{q}}\right\|^{-1}\! +\frac{X}{q^3t}\sum_{\alpha=1}^{q-1}\left\|{\frac{\alpha}{q}}\right\|^{-1}\!+\dfrac{1}{q^{\frac32}}\sum_{\alpha=1}^{q-1}\sum_{\beta=1}^{q-1}\left\|{\frac{\alpha}{q}}\right\|^{-1}\left\|{\frac{\beta}{q}}\right\|^{-1}\right)\notag\\
&\ll_{\epsilon} X^{\epsilon}\left(\dfrac{X}{qt}+\dfrac{t}{q^{1/2}}+q^{1/2}\right)
\end{align}
Notice that if $Z=R$, this first integral vanishes. So we can suppose $Z=q$. With that in mind, if we integrate both sides of inequality \eqref{bound-blomer} against $\dfrac{dt}{t}$, we obtain
\begin{equation}\label{R->Z}
\displaystyle\int_{R}^{Z}\sum_{m\leq\frac{9X}{t^2}}\sum_{\substack{r\leq t\\ r^2m\equiv a\!\!\!\!\pmod{q}}}1\dfrac{dt}{t}\ll_{\epsilon} X^{\epsilon}\left(\dfrac{X}{Rq}+q^{1/2}\right).
\end{equation}
For the remaining integral, we notice that for fixed $m$, the equation $r^2m\equiv a\!\!\pmod{q}$ has at most two solutions for $r$ modulo $q$ (recall that $q$ is prime). Thus, we have (since $Z\geq q$)
\begin{align}\label{Z->sqrtX}
\displaystyle\int_{Z}^{3\sqrt{X}}\sum_{m\leq\frac{9X}{t^2}}\sum_{\substack{r\leq t\\ r^2m\equiv a\!\!\!\!\pmod{q}}}1\dfrac{dt}{t}&\ll \dfrac{X}{Zq}\leq \dfrac{X}{Rq}.
\end{align}
Adding up \eqref{R->Z} and \eqref{Z->sqrtX}, we have, in view of \eqref{SII-int}, that

\begin{equation}\label{SIIfinal}
S_{II}\ll_{\epsilon} X^{\epsilon}\left(R^{-1}Xq^{-1}+q^{1/2}\right).
\end{equation}
Adding together \eqref{I+II}, \eqref{SIfinal} and \eqref{SIIfinal}, we have

\begin{multline}\label{Sfinal}
S-\dfrac{6}{\pi^2}\left(1-\dfrac{1}{q^2}\right)^{-1}\dfrac{X}{q} \ll_{\epsilon}X^{\epsilon}\left(R^{-1}Xq^{-1}+Rq^{-\frac{1}{48}}+R^{\frac{11}{12}}q^{\frac{1}{16}}+q^{\frac{1}{2}}+RY^{-1}\right.\\
\left.+R_0^{-1}XYq^{-{\frac{49}{48}}}+R_0^{-\frac{13}{12}}XYq^{-{\frac{15}{16}}}+R_0^{-2}XYq^{-\frac{1}{2}}+R_0\right).
\end{multline}
Forcing the first, the fifth and the last terms to be equal, we are faced with the choices
$$
R=\left(\frac{X}{q}\right)^{\frac{1}{2}}Y^{\frac{1}{2}},\,R_0=\left(\frac{X}{q}\right)^{\frac{1}{2}}Y^{-\frac{1}{2}}.
$$
Injecting these values in \eqref{Sfinal}, we see that
\begin{multline}\label{Sfinal1}
S-\frac{6}{\pi^2}\left(1-\frac{1}{q^2}\right)^{-1}\frac{X}{q} \ll_{\epsilon}X^{\epsilon}\left(\left(\frac{X}{q}\right)^{\frac{1}{2}}Y^{-\frac{1}{2}}+\left(\frac{X}{q}\right)^{\frac{1}{2}}Y^{\frac{3}{2}}q^{-{\frac1{48}}}\right.\\
\left.+\left(\frac{X}{q}\right)^{\frac{11}{24}}Y^{\frac{37}{24}}q^{\frac{1}{16}}+Y^{2}q^{\frac12}\right). 
\end{multline}
Take $Y=\min\left(q^{\frac1{96}},X^{\frac1{49}}q^{-{\frac{5}{98}}},X^{\frac15}q^{-{\frac25}}\right)$ to optimize the right-hand side of \eqref{Sfinal1}. Then \eqref{S-S0}, \eqref{S0} and \eqref{Sfinal1} imply that we have
$$
E(X,q,a)\ll_{\epsilon}X^{\epsilon}\left(\left(\frac{X}{q}\right)^{\frac{1}{2}}q^{-\frac{1}{192}}+\left(\frac{X}{q}\right)^{\frac{24}{49}}q^{\frac{3}{196}}+\left(\frac{X}{q}\right)^{\frac{2}{5}}q^{\frac{1}{10}}\right),
$$
and the last term is dominated by the first in the range $q\leq X^{\frac{2}{5}}$. Hence we conclude the proof of Theorem \ref{mainFKM}.

\subsection{Proof of Theorem \ref{mainBourg}}
Let $X>1$ and $\eta>0$ be real numbers and let $q$ be a prime number such that $X^{\eta}\leq q\leq X^{{\frac12}-\eta}$. We let again $S$ and $S_0$ be as in the previous section (see \eqref{S=} and \eqref{S0=}). Notice that we have

$$
\left(\frac{X}{q}\right)^{\frac12}\geq q^{\frac{1}{2}+\eta}.
$$
Let $\delta_1>0$ to be chosen later depending on $\eta$. Also let

\begin{equation}\label{RR0Y}
 R=\left(\frac{X}{q}\right)^{{\frac12}+\delta_1},\,  R_0=\left(\frac{X}{q}\right)^{{\frac12}-\delta_1},\, Y=\left(\frac{X}{q}\right)^{2\delta_1}.
\end{equation}
Notice that we can choose $\delta_1$ sufficiently small so that

\begin{equation*}
R_0\geq q^{\frac{1}{2}+{\frac{\eta}2}}. 
\end{equation*}

Theorem \ref{Tbourgain} now gives us a certain $\delta_2>0$ depending on $\eta$ such that

\begin{equation}\label{B-ready}
S(t,q)\leq t^{1-\delta_2},\, (t\geq R_0), 
\end{equation}
where $S(t,q)$ is as in \eqref{Stq}. We start as in the last section, writing

\begin{equation}\label{I+II-3}
S=S_{I}+S_{II},
\end{equation}
where $S_I$ and $S_{II}$ are as in \eqref{R-3}. We deal $S_I$ in the exact same way as before, only replacing each use of Theorem \ref{TFKM} by the upper bound \eqref{B-ready}. Thus we obtain

\begin{multline}\label{SI-B}
 S_I=\dfrac{6}{\pi^2}\left(1-\dfrac{1}{q^2}\right)\dfrac{X}{q}+ O_{\epsilon,\eta}\left(X^{\epsilon}\left(R^{-1}Xq^{-1}+R^{1-\delta_2}+RY^{-1}+R_0^{-1-\delta_2}XYq^{-1}\right.\right.\\
 \left.\left.+R_0^{-1}XYq^{-\frac{3}{2}}+R_0^{-2}XYq^{-\frac{1}{2}}+R_0\right)\right),
\end{multline}
for any $\epsilon>0$ (compare with \eqref{SIfinal}).

As for $S_{II}$, we have the exactly same bound as in the previous case (see \eqref{SIIfinal}). Gathering \eqref{SIIfinal}, \eqref{SI-B} and \eqref{I+II-3}, we see that

\begin{multline*}
 S-\dfrac{6}{\pi^2}\left(1-\dfrac{1}{q^2}\right)\dfrac{X}{q} \ll_{\epsilon,\eta}X^{\epsilon}\left(R^{-1}Xq^{-1}+R^{1-\delta_2}+RY^{-1}+R_0^{-1-\delta_2}XYq^{-1}\right.\\
 \left.+R_0^{-1}XYq^{-\frac{3}{2}}+R_0^{-2}XYq^{-\frac{1}{2}}+R_0\right).
\end{multline*}
Now, we deduce from \eqref{S-S0}, \eqref{S0} and  \eqref{RR0Y} the inequality (recall that $X^{\eta}\leq q\leq X^{{\frac12}-\eta}$)
\begin{multline}
E(X,q,a) \ll_{\epsilon,\eta}X^{\epsilon}\left(\left(\frac{X}{q}\right)^{{\frac12}-\delta_1}+\left(\frac{X}{q}\right)^{{\frac12}+3\delta_1-{\frac{\delta_2}{2}}+\delta_1\delta_2}\right.\\
\left.+\left(\frac{X}{q}\right)^{{\frac12}+3\delta_1-{\frac{\eta}{2}}}+\left(\frac{X}{q}\right)^{{\frac12}+4\delta_1-2\eta}\right).
\end{multline}
Notice that since $q\leq X^{1/2}$, one has $X^{\epsilon}\leq (X/q)^{2\epsilon}$. Now, taking $\epsilon<\delta_1/4$ and $\delta$ sufficiently small, we deduce

$$
E(X,q,a) \ll_{\eta}\left(\frac{X}{q}\right)^{\frac12-\frac{\delta_1}2}.
$$
Taking $\delta:=\delta_1/ 2$ concludes the proof of Theorem \ref{mainBourg}.

\subsection{Proof of Theorem \ref{2/3}}

Let $X>1$ and let $q$ be a prime number such that $q\leq X$. Since the upper bound from Theorem \ref{2/3} is worse than \eqref{hoo-C3} for $q\leq X^{\frac12}$, we can suppose that $q\geq X^{\frac12}$. Let $S$ and $S_0$ be as in \eqref{S=} and \eqref{S0=}, respectively, and $1<R\leq X^{\frac13}$ be a parameter to be chosen optimally later. We split $S$ as before, writing

\begin{equation}\label{I+II-4}
S=S_{I}+S_{II},
\end{equation}
where $S_I$ and $S_{II}$ are as in \eqref{R-3}. For $S_{I}$, it suffices to detect the congruence trivially. We have

\begin{equation}\label{SI4}
S_I=\displaystyle\sum_{\substack{r\leq R\\(r,q)=1}}\mu(r)\left(\dfrac{X}{qr^2}+O(1)\right)= \frac{6}{\pi^2}\prod_{p\mid q}\left(1-\frac{1}{p^2}\right)^{-1}\frac{X}{q} + O(R+ R^{-1}Xq^{-1}).
\end{equation}

For $S_{II}$, by estimating the $\mu$ function trivially and changing the order of summation, we have that

\begin{align}\label{3+4}
|S_{II}|&\leq \displaystyle\sum_{m\leq X/R^2}\displaystyle\sum_{\substack{r\leq (X/m)^{1/2}\\ r^2m\equiv a \!\!\!\! \pmod q}}1\notag\\
&=S_{III}+S_{IV}, 
\end{align}
where

\begin{equation}\label{S3S4}
\begin{cases}
S_{III}=\displaystyle\sum_{m\leq \frac{X}{R^2(\log X)}}\displaystyle\sum_{\substack{r\leq (X/m)^{1/2}\\ r^2m\equiv a \!\!\!\! \pmod q}}1,\\
S_{IV}=\displaystyle\sum_{\frac{X}{R^2(\log X)}< m\leq \frac{X}{R^2}}\displaystyle\sum_{\substack{r\leq (X/m)^{1/2}\\ r^2m\equiv a \!\!\!\! \pmod q}}1.
\end{cases}
\end{equation}

For $S_{III}$, we detect the congruence trivially, obtaining

\begin{align}\label{SIIIfinal}
S_{III}&= \displaystyle\sum_{m\leq \frac{X}{R^2(\log X)}}\left(\frac{X^{\frac12}}{m^{\frac12}q}+O(1)\right)\notag\\
&\ll \frac{X}{Rq(\log X)^{\frac12}} + \frac{X}{R^2(\log X)}.
\end{align}

As for $S_{IV}$, we proceed by dyadic decomposition of the values of $m$. Doing so, we find that there exists $M\geq 1$, a power of two, such that

\begin{equation}\label{ineqforM}
\frac{X}{2R^2(\log X)}< M\leq \frac{2X}{R^2},
\end{equation}
and

\begin{equation}\label{decoupage}
S_{IV}\ll \mathfrak{S}\log\log X,
\end{equation}
where $\mathfrak{S}=\mathfrak{S}(X,M;q,a)$ is given by

\begin{equation}\label{F}
\mathfrak{S}:=\sum_{M<m\leq 2M}\displaystyle\sum_{\substack{r\leq (X/M)^{1/2}\\ r^2m\equiv a \!\!\!\! \pmod q}}1.
\end{equation}
In what follows, we show how to use the Selberg's sieve (see Theorem \ref{SLB}) to estimate $\mathfrak{S}$.

\subsubsection{Implementing the Selberg's sieve}

Let

\begin{equation}\label{a-n}
a_n=\displaystyle\sum_{\substack{M<m\leq 2M\\mn\equiv a\!\!\!\!\pmod{q}}}1,\text{ if }n\leq \frac{X}{M},\\
\end{equation}
and $a_n=0$, otherwise. Then

$$
\mathfrak{S}=\sum_{n=\square}a_n,
$$
where the condition $n=\square$ means that we only sum over the $n$ that are perfect squares.

Let $P$ be a product of distinct odd primes such that $p\nmid q$. For each $p\mid P$, let $\Omega_p$ denote the set of non-square residue classes modulo $p$. Note that we can soften the condition $n=\square$ to $n\not\in\Omega_p$ for every $p\mid P$. In other words, the following inequality holds:

$$
\mathfrak{S}\leq \sum_{\substack{n\not\in \Omega_p\\ \text{for every }p\mid P}}a_n. 
$$
We want to use Theorem \ref{SLB}. Thus, we need to give asymptotic formulas for $|\mathcal{A}_d|$ (see \eqref{Ad}). We notice that with $a_n$ as in \eqref{a-n}, we have

\begin{equation}\label{A=sumG}
|\mathcal{A}_d|=\sum_{\substack{\alpha\!\!\!\!\pmod{n}\\\alpha\!\!\!\!\pmod{p}\in\Omega_{p}\\ \text{for every }p\mid d}}\mathcal{G}\left(d,\alpha\right), 
\end{equation}
where

\begin{align}\label{Gda}
\mathcal{G}\left(d,\alpha\right):=\sum_{n\equiv\alpha \!\!\!\!\pmod{d}}a_n.
\end{align}
We use the $\psi$ function to evaluate $\mathcal{G}\left(d,\alpha\right)$. Injecting \eqref{a-n} in \eqref{Gda} and interchanging the order of summation, gives

\begin{equation*}
\mathcal{G}\left(d,\alpha\right)=\sum_{M<m\leq 2M}\left(\dfrac{X}{Mdq}-\psi\left(\dfrac{X}{Mdq}-\dfrac{a{\overline{dm}}}{q}-\dfrac{\alpha\overline{q}}{d}\right)+\psi\left(-\dfrac{a{\overline{dm}}}{q}-\dfrac{\alpha\overline{q}}{d}\right)\right),
\end{equation*}
for every $d$ coprime with $q$ and $\alpha\!\!\pmod{d}$. By Lemma \ref{boundpsi}, we obtain that

\begin{equation}\label{G-3}
\mathcal{G}\left(d,\alpha\right)=\dfrac{X}{dq}+O\left(\frac{M\log q(\log\log q)^{4}}{(\log M)^{\frac12}}\right).
\end{equation}
Note that the inequality \eqref{ineqforM} implies
$$
M\geq X^{\frac13}(\log X)^{-1}.
$$
Hence, \eqref{G-3} and \eqref{A=sumG} imply

\begin{equation}\label{Ad=}
|\mathcal{A}_d|=g(d)\dfrac{X}{q}+O\left(d\frac{M(\log\log X)^{4}}{(\log X)^{\frac12}}\right), 
\end{equation}
where $g(d)$ is the multiplicative function supported on squarefree numbers and such that

$$
g(p)=\dfrac{p-1}{2p}.
$$
Let $D>1$ and
$$
P:=\prod_{\substack{2<p\leq \sqrt{D}\\p\nmid q}}p.
$$
We use Theorem \ref{SLB} for $Y=\frac{X}{q}$ and $D$ and $P$ as above. We obtain, in view of \eqref{Ad=}, the inequality 

$$
\mathfrak{S}\leq \sum_{\substack{n\not\in \Omega_p\\ \text{for every }p\mid P}}a_n\ll \frac{X}{q}J^{-1} + \frac{M(\log\log X)^{4}}{(\log X)^{\frac12}}\sum_{d\leq \sqrt{D}}\tau_3(d)d,
$$
where

$$
J=\sum_{\substack{d\mid P\\d<\sqrt{D}}}h(d)\geq \sum_{\substack{p<\sqrt{D}\\p\nmid q}}\frac{g(p)}{1-g(p)}\gg \frac{\sqrt{D}}{\log D}.
$$
Thus

\begin{equation}\label{AfterJ}
\mathfrak{S}\ll \frac{X\log D}{q\sqrt{D}}+ \frac{DM(\log D)^2(\log\log X)^{4}}{(\log X)^{\frac12}}. 
\end{equation}
Gathering \eqref{ineqforM}, \eqref{decoupage} and \eqref{AfterJ}, we see that

$$
S_{IV}\ll \frac{X\log D\log\log X}{D^{\frac12}q} +\frac{DX(\log D)^2(\log\log X)^5 }{ R^2(\log X)^{\frac12}}.
$$
For each $0<\gamma< {\frac12}$, the choice $D=(\log X)^{\gamma}$ gives the inequality

\begin{equation}\label{SIVfinal}
 S_{IV}\ll \frac{X(\log\log X)^2}{q(\log X)^{\frac{\gamma}{2}}} +\frac{X(\log\log X)^7}{R^2(\log X)^{{\frac12}-\gamma}}.
\end{equation}
Putting together \eqref{I+II-4}, \eqref{SI4}, \eqref{3+4}, \eqref{SIIIfinal} and \eqref{SIVfinal}, we obtain

\begin{equation}\label{beforeR}
S-\dfrac{6}{\pi^2}\left(1-\dfrac{1}{q^2}\right)^{-1}\dfrac{X}{q}\ll R + R^{-1}Xq^{-1} +\frac{X(\log\log X)^2}{q(\log X)^{\frac{\gamma}{2}}}+\frac{X(\log\log X)^7}{R^2(\log X)^{{\frac12}-\gamma}}. 
\end{equation}
Forcing the first and last terms to be equal, we are faced with the choice

$$
R=X^{\frac13}(\log X)^{-{\frac16}+{\frac{\gamma}3}}(\log\log X)^{\frac73}.
$$
Replacing it in \eqref{beforeR} gives the inequality

\begin{equation}\label{afterR}
 S-\dfrac{6}{\pi^2}\left(1-\dfrac{1}{q^2}\right)^{-1}\dfrac{X}{q}\ll \frac{X^{\frac13}(\log\log X)^{\frac73}}{(\log X)^{{\frac16}-{\frac{\gamma}3}}}+\frac{X(\log\log X)^2}{q(\log X)^{\frac{\gamma}2}}.
\end{equation}
Theorem \ref{2/3} now follows by combining \eqref{S-S0}, \eqref{S0} and \eqref{afterR}.

\section*{acknowledgements}
This work was done as a part of my Ph.D. thesis under the guidance of Étienne Fouvry. I am very grateful for his suggestions and valuable remarks.


\end{document}